\documentclass[11pt,reqno]{amsart}
\usepackage{amssymb}
\usepackage{amsthm}
\usepackage{eucal}
\usepackage{a4wide}

\newcommand{\R}{\mathbb R}

\newcommand{\Z}{\mathbb Z}

\newcommand{\C}{\mathbb C}

\newcommand{\secao}[1]{\section{#1}\setcounter{equation}{0}}

\newtheorem{theorem}{Theorem}[section]
\newtheorem{proposition}[theorem]{Proposition}
\newtheorem{remark}[theorem]{Remark}
\newtheorem{lemma}[theorem]{Lemma}
\newtheorem{corollary}[theorem]{Corollary}
\newtheorem{definition}[theorem]{Definition}
\newtheorem{example}[theorem]{Example}
\begin{document}
%\today
\title[Supercritical KdV]{On the supercritical KdV equation with time-oscillating nonlinearity}

\author{M. Panthee}
\thanks{M. P. was partially supported by  the Research Center of Mathematics of the University of Minho, Portugal through the FCT Pluriannual Funding Program, and through the project PTDC/MAT/109844/2009, and M. S.  was partially supported by FAPESP Brazil.}
\address{Centro de Matem\'atica, Universidade do Minho\\ 4710--057, Braga, Portugal.}
\email{mpanthee@math.uminho.pt}
\author{M. Scialom}

\address{IMECC-UNICAMP\\
13083-970, Campinas, S\~ao Paulo, Brazil}
\email{scialom@ime.unicamp.br}

\begin{abstract} For the initial value problem (IVP) associated the generalized Korteweg-de Vries (gKdV) equation with supercritical nonlinearity,
\begin{equation*}
u_{t}+\partial_x^3u+\partial_x(u^{k+1}) =0,\qquad k\geq 5,
\end{equation*}
 numerical evidence \cite{BDKM1, BSS1} shows that, there are initial data  $\phi\in H^1(\mathbb{R})$  such that the corresponding solution  may blow-up in finite time. Also, with the evidence from numerical simulation \cite{ACKM, KP}, the physicists claim that a periodic time dependent term in factor of the nonlinearity would disturb the blow-up solution, either accelerating  or delaying it.
 
In this work, we investigate  the IVP associated to the gKdV equation
\begin{equation*}
u_{t}+\partial_x^3u+g(\omega t)\partial_x(u^{k+1}) =0, 
\end{equation*}
 where $g$ is a periodic function and $k\geq 5$ is an integer. We prove that, for given initial data $\phi \in H^1(\R)$, as $|\omega|\to \infty$,  the solution $u_{\omega}$ converges to the solution $U$ of the initial value problem associated to
\begin{equation*}
U_{t}+\partial_x^3U+m(g)\partial_x(U^{k+1}) =0,
\end{equation*}
with the same initial data, where $m(g)$ is the average of the periodic function $g$. Moreover, if the solution $U$ is global and satisfies $\|U\|_{L_x^5L_t^{10}}<\infty$, then we prove that the solution $u_{\omega}$ is also global provided $|\omega|$ is sufficiently large.

\end{abstract}

\maketitle

%%%%%%%%%%%%%%%%%%%%%%%%%%
\secao{Introduction}
%%%%%%%%%%%%%%%%%%%%%%%%%%

Motivated from our earlier work in \cite{CPS1} for the critical KdV equation, we  consider the initial value problem (IVP)
\begin{equation}\label{ivp1}
\begin{cases}
u_{t}+\partial_x^3u+g(\omega t)\partial_x(u^{k+1}) =0,\\
u(x,t_0)=\phi(x),
\end{cases}
\end{equation}
where $x,\, t, t_0, \omega\in \R$ and $u=u(x,t)$ is a real valued function, $k\geq 5$ is an integer  and $g\in C(\R, \R)$ is a periodic function
with period $L>0$. To simplify the analysis, we translate the initial time $t_0$ to $0$ and consider the following IVP
\begin{equation}\label{ivp1.1}
\begin{cases}
u_{t}+\partial_x^3u+g(\omega(t+t_0))\partial_x(u^{k+1}) =0,\\
u(x,0)=\phi(x).
\end{cases}
\end{equation}

Before analyzing the  IVP \eqref{ivp1} with time oscillating nonlinearity, we discuss
some aspects of the supercritical  Korteweg-de Vries (KdV) equation,
\begin{equation}\label{ivp2}
\begin{cases}
u_{t}+\partial_x^3u+ \partial_x(u^{k+1}) =0,\qquad k\geq 5,\\
u(x,0)=\phi(x),\qquad x,\, t\in \R.
\end{cases}
\end{equation}

In the literature, the equation \eqref{ivp2} is known as the supercritical KdV equation because, if one considers the  nonlinearity $\partial_x(u^{k+1})$, $k\in \Z$, then the case when $k =4$  is the critical  one. As described in \cite{KPV3}, the case $k=4$ is called critical for three different reasons. First one is that the global solution exists for all data  in $H^1(\R)$,  whenever $k=1,2,3$. While for $k=4$ the global solution exists only for small data (i.e., data with small $H^1(\R)$-norm). Second reason is that the index $k=4$ is critical for the orbital stability of the solitary wave solutions, see \cite{BSS1}. Finally, the third reason is that the case $k=4$ is the only power for which a solitary wave solution cannot have arbitrarily small $L^2$-norm, see \cite{KPV3}.

Well-posedness issues for the IVP \eqref{ivp2} have been extensively  studied in the literature, see for example \cite{Kato1} and \cite{KPV3}, \cite{KPV4} and references therein. A detailed account of the recent well-posedness results can be found in Kenig, Ponce and Vega \cite{KPV3}, where they proved that, there exists $\delta_k >0$ such that  the IVP \eqref{ivp2} is globally well-posed for any data $ \phi\in H^s(\R)$, $s\geq s_k:=\frac12-\frac k2$ satisfying $\|D_x^{s_k}\phi\|_{L_x^2}<\delta_k$.  They were also able to relax the smallness condition on the given data to obtain local well-posedness result, but paying the price that the existence time now depends on the shape of the data $\phi$ and not just in its size. These are the best well-posedness results in the sense that $s=s_k$ is the critical exponent given by the scaling argument. However, for data in $H^s(\R)$, $s>s_k$, they were able to remove the size and shape restriction  and got local-well posedness for arbitrary data with life span $T$ of the solution depending on $\|\phi\|_{H^s(\R)}$. Quite recently, Farah et. al. \cite{FLP3} considered the IVP \eqref{ivp2} to deal global well-posedness for the data with low Sobolev regularity. In this context, they proved the following local well-posedness result in the function space slightly different from the one used in \cite{KPV3}. In what follows, we state this result, because we will modify it to suit in our context later on. 

\begin{theorem}\label{th-flp1}\cite{FLP3} Let $k>4$ and $s>s_k:=\frac12-\frac k2$. Then for any $\phi\in H^s(\R)$ there exist $T=T(\|\phi\|_{H^s(\R)})>0$ (with $T(s, \rho) \to \infty$ as $\rho \to 0$) and a unique strong solution $u$  to the IVP \eqref{ivp2} satisfying:
\begin{equation}\label{ps.0}
u \in C([0, \,T]; H^s(\R)),
\end{equation}
\begin{equation}\label{ps.1}
\|\partial_xu\|_{L_x^{\infty}L_T^2} +\|D_x^s\partial_xu\|_{L_x^{\infty}L_T^2} < \infty,
\end{equation}
\begin{equation}\label{ps.2}
\|u\|_{L_x^{5}L_T^{10}} +\|D_x^su\|_{L_x^{5}L_T^{10}} < \infty, 
\end{equation}
\begin{equation}\label{ps.3}
\|D_t^{\gamma_k}D_x^{\alpha_k}D_t^{\beta_k}u\|_{L_x^{p_k}L_T^{q_k}} < \infty,
\end{equation}
where
\begin{equation}\label{ps.5}
\alpha_k= \frac1{10}-2{5k}, \qquad \beta_k= \frac3{10}-\frac6{5k}, \qquad \gamma_k= \gamma_k(s)=\frac{s-s_k}3
\end{equation}
\begin{equation}\label{ps.6}
\frac1{p_k} =\frac2{5k}+\frac1{10}, \qquad \frac1{q_k}= \frac3{10}-\frac4{5k}.
\end{equation}

Moreover, for any $T' \in (0, T)$, there exists a neighborhood
 $\mathcal{V}$ of $\phi$ in $H^s(\R)$ such that the map
 $\tilde{\phi}\mapsto \tilde{u}$ from
 $\mathcal{V}$ into the class defined by \eqref{ps.0} to \eqref{ps.3} with
 $T'$ in place of $T$ is Lipschitz.
\end{theorem}

We recall that, the $L_x^2(\R)$ norm and energy are conserved by the flow of \eqref{ivp2}. More precisely, 
\begin{equation}\label{c.11}
\int_{\R} |u(x,t)|^2dx = \int_{\R} |\phi(x)|^2dx,
\end{equation}
and
\begin{equation}\label{c.12}
E(u(\cdot, t)):= \frac12\int_{\R}\{ (u_x(x,t))^2 - c_ku^{k+2}(x,t)\}dx = E(\phi),
\end{equation}
are conserved quantities.

As shown with a detailed calculations  in \cite{FLP3}, these conserved quantities yield an {\em a priori} estimate for $\|\partial_xu(t)\|_{L^2(\R)}$ if the initial data $\phi$ is sufficiently small in $H^1(\R)$. This allows to iterate the local  solution to get the global one for small data in $H^1(\R)$.  However, a numerical study carried out by Bona et. al. \cite{BDKM1, BDKM2} (see also \cite{BSS1}) revel the existence of $H^1$-data for which the corresponding solution to the supercritical KdV equation may blow-up in finite time. This is the point that motivated us to carry on this work in the light of the recent work of Abdullaev et. al. in \cite{ACKM} and Konotop and Pacciani in \cite{KP}. 

The authors in  \cite{ACKM} and \cite{KP} investigate the effect of a time oscillating term in factor of the nonlinearity in Bose-Einstein condensates. In \cite{ACKM} the authors investigate solutions which are global for large frequencies, while the authors in \cite{KP} study solutions which blow-up in
finite time. Their results are numerical. Roughly speaking, they claim that the periodic time dependent term in factor of the nonlinearity would disturb the blow-up solution, either accelerating it or delaying it. Recently, Cazenave and Scialom \cite{CS} considered the nonlinear Schr\"odinger (NLS) equation and got an analytical insight to understand the problem  by showing that the solution really depends on the frequency of the oscillating term. They proved that the solution $u$ to the IVP associated to the NLS equation 
\begin{equation}\label{cs.1}
iu_t+\Delta u +\theta(\omega t)|u|^{\alpha}u = 0, \quad x\in \R^N,
\end{equation}
where $0<\alpha<\frac4{(N-2)^+}$ is an $H^1$ sub-critical exponent and $\theta$ is a periodic function, with initial data $\phi\in H^1(\R^N)$ converges as $|\omega|\to \infty$ to the solution $U$ of the limiting equation
\begin{equation}\label{cs.2}
iU_t+\Delta U +I(\theta)|U|^{\alpha}U = 0, \quad x\in \R^N,
\end{equation}
with the same initial data, where $I(\theta)$ is the average of $\theta$. Moreover, they also showed that, if the limiting solution $U$ is global and has a certain decay property as $t \to \infty$, then $u$ is also global if $|\omega|$ is sufficiently large. A similar result has been proved for the critical KdV equation in our earlier work \cite{CPS1}.

In this work, we are interested in obtaining similar results for the supercritical KdV equation.
The numerical evidences  for the existence of blow-up solution to \eqref{ivp2} in $H^1(\R)$ due to Bona et. al. \cite{BDKM1, BDKM2} (see also \cite{BSS1}) and the discussion made above strengthen our motivation of studying \eqref{ivp1} with time oscillating nonlinearity.

As discussed above, our interest here is to investigate the behavior  in $H^1(\R)$ of the solution of the IVP \eqref{ivp1} as $|\omega| \to \infty$. The natural limiting candidate to think of is the solution to the following IVP
\begin{equation}\label{ivp3}
\begin{cases}
U_{t}+\partial_x^3U+m(g)\partial_x(U^{k+1}) =0,\\
U(x,0)=\phi(x),\qquad x,\, t\in \R,
\end{cases}
\end{equation}
where $m(g):= \frac1L\int_0^Lg(t)dt$ is the mean value of $g$ and is a real number.
To this end, we need an appropriate well-posedness result for the supercritical KdV equation in $H^1(\R)$. We recall the local well-posedness result from \cite{FLP3} for arbitrary data in $H^s(\R)$, $s>s_k$, with life span of solution depending only on the $H^s(\R)$-norm of the initial data stated in Theorem \ref{th-flp1} (See also  \cite{KPV3}). The function space used in Theorem \ref{th-flp1} has additional norm $\|D_t^{\gamma_k}D_x^{\alpha_k}D_t^{\beta_k}u\|_{L_x^{p_k}L_T^{q_k}}$ that involves time derivatives of the solution. The presence of these norms create extra difficulty to handle the time-oscillating  nonlinearity. Therefore, to deal with our case, we need to avoid the presence of the norm that involved time derivatives. Also, it is very important to have an explicit expression that gives the local existence time of the solution.  In the literature, we did not find an explicitly written proof of the $H^1(\R)$ well-posedness for the IVP \eqref{ivp2}that fulfills our requirement. Therefore, we will provide a new proof for the well-posedness of the IVP \eqref{ivp2} in $H^1(\R)$. Our proof allows us to extend the result to \eqref{ivp1.1} and as a consequence to have an estimate of the local existence time.

The only works other than \cite{CPS1} and \cite{CS} we did find in the literature that address the well-posedness issue for the equations of the KdV-family and NLS with explicitly time dependent nonlinearity were  by Nunes \cite{Nunes-1, Nunes-2} and Damergi and Goubet \cite{DG}. The authors in  \cite{DG}  deal with the NLS equation in $\R^2$ with nonlinearity $\cos^2(\Omega t)|u|^{p-1}u$ in the critical and supercritical cases. The author in \cite{Nunes-1}  considered the transitional KdV with nonlinearity $f(t) u \partial_xu$,  $f$ a continuous function such that $f'\in L_{\mbox{loc}}^1(\R)$ and proved global well-posedness in $H^s(\R)$, $s\geq 1$. The transitional KdV arises in the study of long solitary waves propagating on the thermocline separating two layers of fluids of almost equal densities in which the effect of the change in the depth of the bottom layer, which the wave feels as it approaches the shore, results in the coefficient of the nonlinear term, for details see \cite{KN}. In \cite{Nunes-2}, transitional Benjamin-Ono equation with time dependent coefficient in the nonlinearity has been considered and the main result is the global existence of the solution for data in $H^s(\R)$, $s\geq \frac32$.

Before stating the main results of this work, we define notations that will be used throughout this work.

\noindent
{\bf{Notation}}:
 We  use $\hat{f}$ to denote  the Fourier transform of $f$ and is defined  as,
\begin{equation*}
\hat{f}(\xi) = \frac{1}{(2\pi)^{1/2}}\int_{\R} e^{-ix\xi}f(x)\,dx.
\end{equation*} The $L^2$-based Sobolev space of order $s$ will be denoted  by $H^s$
with norm
\begin{equation*}
\|f\|_{H^s(\R)} = \Big(\int_{\R} (1+\xi^2)^s|\hat{f}(\xi)|^2\,d\xi\Big)^{1/2}.
\end{equation*}
The Riesz potential of order $-s$ is denoted by $D_x^s =
 (-\partial_x^2)^{s/2}$. For $f:\R\times [0, T] \to \R$ we define the mixed
 $L_x^pL_T^q$-norm by
\begin{equation*}
\|f\|_{L_x^pL_T^q} = \Big(\int_{\R}\Big(\int_0^T |f(x, t)|^q\,dt
\Big)^{p/q}\,dx\Big)^{1/p},
\end{equation*}
with usual modifications when $p = \infty$. We replace $T$ by $t$ if $[0, T]$ is the whole real line $\R$. We use the notation $f\in H^{\alpha+}$ if $f\in H^{\alpha+\epsilon}$ for $\epsilon >0$.

We define two more spaces $X_T$ and $Y_T$ with norms  
\begin{equation}\label{n.11}
\begin{split}
 \|f\|_{X_T}  :=&\|f\|_{L_T^{\infty}H^1}+ \|\partial_xf\|_{L_x^{\infty}L_T^2} +
\|\partial_x^2f\|_{L_x^{\infty}L_T^2}\\
& + \|f\|_{L_x^{5}L_T^{10}}+\|\partial_xf\|_{L_x^{5}L_T^{10}} + 
\|\partial_xf\|_{L_x^{20}L_T^{5/2}}+
\|f\|_{L_x^{4}L_T^{\infty}},
\end{split}
\end{equation}
and
\begin{equation}\label{n.21}
\|f\|_{Y_T} := \|\partial_xf\|_{L_x^2L_T^2}+ \|f\|_{L_x^2L_T^2},
\end{equation}
respectively. We replace $X_T$ by $X_t$ or $X_{(T, \infty)}$, if the time integral is taken in the interval $(0, \infty)$ or $(T, \infty)$ respectively, and similarly for  $Y_T$.

We use the letter $C$ to denote various  constants whose exact values are immaterial and which may vary from one line to the next.

First, let us state the $H^1$-local well-posedness result for the IVP \eqref{ivp2} in a function space that does not use norms involving time derivatives of the solution.

\begin{theorem}\label{th.1}
Suppose $\phi \in H^1(\R)$. Then there exist $T = T(\|\phi\|_{H^1(\R)})>0 $ and a unique solution
$u$  to the IVP \eqref{ivp2} satisfying
\begin{equation}\label{mmx.0}
u \in C([0, \,T]; H^1(\R)),
\end{equation}
\begin{equation}\label{mmx.1}
\|\partial_xu\|_{L_x^{\infty}L_T^2} +\|\partial_x^2u\|_{L_x^{\infty}L_T^2} < \infty,
\end{equation}
\begin{equation}\label{mmx.3}
\|u\|_{L_x^{5}L_T^{10}} +\|\partial_xu\|_{L_x^{5}L_T^{10}} +\|\partial_xu\|_{L_x^{20}L_T^{5/2}}< \infty, 
\end{equation}
\begin{equation}\label{mmx.5}
\|u\|_{L_x^{4}L_T^{\infty}} < \infty.
\end{equation}

Moreover, for any $T' \in (0, T)$, there exists a neighborhood
 $\mathcal{V}$ of $(u_0, v_0)$ in $H^1(\R)$ such that the map
 $\tilde{\phi}\mapsto \tilde{u}$ from
 $\mathcal{V}$ into the class defined by \eqref{mmx.0} to \eqref{mmx.5} with
 $T'$ in place of $T$ is Lipschitz.
\end{theorem}

 Using Duhamel's principle, we prove Theorem \ref{th.1} by considering  the
integral equation associated to the IVP \eqref{ivp2},
\begin{equation}\label{equint1}
u(t)= S(t)\phi - \int_{0}^{t}S(t-t')\partial_x(u^{k+1})(t')\,dt',
\end{equation}
where $S(t)$ is the unitary group generated by the operator $\partial_x^3$ that describes the solution to the linear problem. Our interest is to solve \eqref{equint1} using the contraction mapping principle in appropriate metric spaces.

\begin{remark}
Since the average $m(g)$ of $g$ is a constant, the proof of Theorem \ref{th.1} can be adapted line by line to obtain the similar well-posedness result for the IVP \eqref{ivp3}. The only difference in this case is that, to complete the contraction argument we need to choose $T>0$ in such a way that $C|m(g)|T^{1/2}\|\phi\|_{H^1(\R)}^4<\frac12$.  So the existence time $T$ depends on $|m(g)|$ and $\|\phi\|_{H^1(\R)}$. We also have the following bound
\begin{equation}\label{pr.2}
\|U\|_{X_T}\leq C\|\phi\|_{H^1(\R)}, \quad \forall\;\, t\in [0, T].
\end{equation}

\end{remark}

Regarding the well-posedness results for the IVP \eqref{ivp1.1}, we have the following theorem.

\begin{theorem}\label{th.2}
Suppose $\phi \in H^1(\R)$. Then there exist $T = T(\|\phi\|_{H^1(\R)}, \|g\|_{L^{\infty}})>0 $ and a unique solution
$u_{\omega, t_0}\in C([0, \,T]; H^1(\R))$  to the IVP \eqref{ivp1.1} satisfying \eqref{mmx.1}--\eqref{mmx.5}.

Moreover, for any $T' \in (0, T)$, there exists a neighborhood
 $\mathcal{V}$ of $\phi$ in $H^1(\R)$ such that the map
 $\tilde{\phi}\mapsto \tilde{u}_{\omega, t_0}$ from
 $\mathcal{V}$ into the class defined by \eqref{mmx.0} to \eqref{mmx.5} with
 $T'$ in place of $T$ is Lipschitz.
\end{theorem}

Now, we state the main results of this work.
\begin{theorem}\label{main.th1}
Fix $\phi \in H^1(\R)$. For given $\omega, t_0 \in \R$, let $u_{\omega, t_0}$ be the maximal solution of the IVP \eqref{ivp1.1} and $U$ be the solution of the limiting IVP \eqref{ivp3} defined on the maximal time of existence $[0, S_{\max})$. Then, for given any $0<T<S_{\max}$, the solution $u_{\omega, t_0}$ exists on $[0, T]$ for all $t_0\in \R$ and $|\omega|$ large. Moreover, $\|u_{\omega, t_0}-U\|_{X_T} \to 0$, as $|\omega|\to \infty$, uniformly in $t_0\in \R$. In particular, the convergence holds in $C([0, T]; H^1(\R))$ for all $T\in (0, S_{max})$.
\end{theorem}

\begin{theorem}\label{main.th2}
Let  $\phi \in H^1(\R)$ and $u_{\omega, t_0}$ be the maximal solution of the IVP \eqref{ivp1}. Suppose $U$ be the maximal solution of the IVP \eqref{ivp3} defined on $[0, S_{max})$. If $S_{max}=\infty$ and
\begin{equation}\label{eq1.20}
\|U\|_{L_x^{5}L_t^{10}} < \infty,
\end{equation}
then it follows that $u_{\omega, t_0}$ is global for all $t_0 \in \mathbb{R}$ if $|w|$ is sufficiently large. Moreover,
\begin{equation}\label{eq1.19}
\|u_{\omega, t_0}-U\|_{X_t} \to 0, \quad \textrm{when}\quad |w| \to \infty,
\end{equation}
uniformly in $t_0$. In particular, convergence holds in $L^{\infty}((0,\infty); H^1(\R))$.
\end{theorem}

In view of the numerical prediction in \cite{BDKM1, BDKM2} (see also \cite{BSS1})  of existence of blow-up solution for the supercritical KdV equation for in $H^1(\R)$, the Theorem \ref{main.th2} is very interesting in the sense that when $m(g)=0$ the solution $U$ to the IVP \eqref{ivp3} will be global for all initial $H^1$-data and the  solution $u_{\omega, t_0}$ to the nonlinear problem \eqref{ivp1.1} will be global too, for $|\omega|$ large enough.
 
 Before leaving this section, we discuss the example constructed in \cite{CS} in the context of the NLS equation with time oscillating nonlinearity. The authors in \cite{CS} showed that for small frequency $|\omega|$, the solution $u_{\omega, t_0}$ blows-up in finite time or is global depending on $t_0$, while for the large frequency $|\omega|$, the solution $u_{\omega, t_0}$  is global for all $t_0\in \R$. The same example can be utilized with small modification  in the context of the critical KdV equation. We present it here for the convenience of the readers.

 \begin{example}
 Let $L>1$,  $0<\epsilon < \frac{L-1}2$ and consider a periodic function $g$ defined by
 \begin{equation}\label{example}
 m(g)=0, \quad\mbox{and}\quad g(s) = \begin{cases} 1, \quad |s|\leq \epsilon,\\
                                                   0, \quad 1\leq s \leq 1+\epsilon,
                                      \end{cases}
\end{equation}
 with period $L$.

 Fix $\phi \in H^1(\R)$ and assume that the solution $v$ of the IVP
 \begin{equation}\label{x-ivp}
 \begin{cases}
 v_t+v_{xxx}+v^{k+1}\partial_xv=0, \qquad k\geq 5,\\
 v(x,0) = \phi(x),
 \end{cases}
 \end{equation}
 blows-up in finite time, say $T^*$. In the light of the numerical evidences presented in \cite{BDKM1, BDKM2} (see also \cite{BSS1}) we can suppose that  such a solution $v(x,t)$ of \eqref{x-ivp} with $t\in [0, T^*)$, exists.

 From Theorem \ref{main.th1}, for this particular $\phi$ and the periodic function $g$,  we have that the solution $u_{\omega, t_0}$ to the IVP \eqref{ivp1.1} converges, as $|\omega|\to \infty$, to the solution $U$ of the linear KdV equation with same initial data $\phi$. So, in view of Theorem \ref{main.th2}, $u_{\omega, t_0}$ is global as $|\omega|\to \infty$ for all $t_0\in \R$.

Now we move to analyze the behavior of the solution for $|\omega|$ small. Note that $g(\omega s) =1$ when $|\omega s|\leq \epsilon$. Therefore, if we consider $|\omega| < \frac{\epsilon}{T^*}$, then we see that the solution $v$ to the IVP \eqref{x-ivp} satisfies \eqref{ivp1.1} for $t_0=0$ on $[0, T^*)$. By uniqueness, $u_{\omega, 0} =v$. Hence the solution $u_{\omega, 0}$ of the IVP \eqref{ivp1.1} blows-up in finite time, provided $|\omega|<\frac{\epsilon}{T^*}$.

Let $\epsilon = \epsilon(A)$ be as in Corollary \ref{xavcor} with $A=\|g\|_{L_t^{\infty}}$. From the linear estimate \eqref{241} we have that $S(\cdot)\phi\in L_x^5L_t^{10}$, so there exists $T>0$ such that
\begin{equation}\label{em.1}
\|S(\cdot)[S(T)\phi]\|_{L_x^5L_t^{10}} = \|S(\cdot)\phi\|_{L_x^5L_{(T, \infty)}^{10}} \leq \epsilon.
\end{equation}

For $\omega > 0$, if we consider $t_0 =\frac1\omega$, we have that $g(\omega(s+t_0)) = 0$ for all $1\leq \omega(s+t_0)\leq 1+\epsilon$, i.e., for all $0\leq s\leq \frac\epsilon{\omega}$. Therefore, if we let $\omega> 0$ satisfying $\omega\leq\frac\epsilon T$ (i.e., $T\leq\frac\epsilon{\omega}$), and choose $t_0 =\frac1\omega$, then $g(\omega(s+t_0)) = 0$ for all $0\leq s\leq T$. So, with this choice, $u_{\omega, t_0}$ solves the linear KdV equation if $0\leq t\leq T$. Therefore, for  $\omega\leq\frac\epsilon T$, $u_{\omega, t_0}$ exists on $[0, T]$ and is given by $S(t)\phi$, in particular $u_{\omega, t_0}(T)=S(T)\phi$. From \eqref{em.1}, $\|S(\cdot)u_{\omega, t_0}(T)\|_{L_x^5L_t^{10}} \leq \epsilon$. Hence, from Corollary \ref{xavcor} we  conclude that $u_{\omega, t_0}$ is global.
\end{example}

This paper is organized as follows. In Section \ref{linear1} we record some
preliminary estimates associated to the linear problem and other relevant
results. In Section \ref{prmt} we give a proof of the local well-posedness
result for the supercritical KdV equation in  $H^1(\R)$ and some other results that will be used in the proof of the main Theorems. Finally, the proof
of the main results  will be given in Section \ref{prmt-2}.

%%%%%%%%%%%%%%%%%%%%%%%%%%%%%%%%%%%%%%%%%%%%%%%%%%%%%%%%%%%
\secao{Preliminary estimates}\label{linear1}
%%%%%%%%%%%%%%%%%%%%%%%%%%%%%%%%%%%%%%%%%%%%%%%%%%%%%%%%%%%

In this section we record some linear estimates associated to the IVP (\ref{ivp1}).  These estimates are not new and can be found in the literature. For the sake of clearness we sketch the ideas involved and provide references where a detailed proof can be found.

\begin{lemma}\label{lema2.1}
If $u_0 \in L^2(\R)$, then
\begin{equation}\label{221}
\|\partial_xS(t)u_0\|_{L_x^{\infty}L_t^2} \leq C\|u_0\|_{L_x^2}.
\end{equation}
If $f \in L_x^{1} L_t^2$, then
\begin{equation}\label{eq2.2}
\Big\|\partial_x\int_0^{t}S(t-t')f(\cdot,t')dt'\Big\|_{L_t^{\infty}L_x^2} \leq C\|f\|_{L_x^{1} L_t^2},
\end{equation}
and
\begin{equation}\label{eq2.3}
\Big\|\partial_x^{2}\int_0^{t}S(t-t')f(\cdot,t')dt'\Big\|_{L_x^{\infty}L_t^2} \leq C\|f\|_{L_x^{1} L_t^2}.
\end{equation}
\end{lemma}
\begin{proof}
 For the proof of the homogeneous smoothing effect (\ref{221}) and the double smoothing effect \eqref{eq2.3}, see Theorem 3.5 in \cite{KPV3}  (see also Section $4$ in  \cite{KPV5}). The inequality \eqref{eq2.2} is the dual version of \eqref{221}.
\end{proof}

Now we give the maximal function estimate.

\begin{lemma}\label{lema2.2}
If $u_0 \in \dot{H}^{1/4}(\R)$, then
\begin{equation}\label{231}
\|S(t)u_0\|_{L_x^4L_T^{\infty}} \leq C \|D_x^{1/4}u_0\|_{L^2(\R)}.
\end{equation}
 Also, we have
\begin{equation}\label{231.1}
\|S(t)u_0\|_{L_x^{\infty}L_T^{\infty}} \leq C \|u_0\|_{H^{\frac12+}(\R)}.
\end{equation}
%and for, $0\leq\theta\leq 1$
%\begin{equation}\label{231.2}
%\|S(t)u_0\|_{L_x^{\frac4\theta}L_T^{\infty}} \leq C \|u_0\|_{H^{\frac{2-\theta}4+}(\R)}.
%\end{equation}
\end{lemma}
\begin{proof}
For the proof of the estimate \eqref{231} we refer to Theorem 3.7 in \cite{KPV3} (see also \cite{KPV6} and  \cite{KR1}). The estimate \eqref{231.1} follows from Sobolev embedding. 
%The interpolation between \eqref{231} and \eqref{231.1}  yields \eqref{231.2}.
\end{proof}

In what follows, we state some more estimates that will be used in our analysis.
\begin{lemma}\label{lema2.3}
If $u_0 \in L^2(\R)$, then
\begin{equation}\label{241}
\|S(t)u_0\|_{L_x^{5}L_t^{10}} \leq C \|u_0\|_{L_x^2}.
\end{equation}
Also we have
\begin{equation}\label{242}
\| \partial_xS(t)u_0\|_{L_x^{20}L_t^{5/2}} \leq C \|D_x^{1/4}u_0\|_{L_x^2},
\end{equation}
and 
\begin{equation}\label{242-1}
\| \partial_xS(t)u_0\|_{L_x^{40/3}L_t^{20/7}} \leq C \|D_x^{3/8}u_0\|_{L_x^2}.
\end{equation} 
\end{lemma}
\begin{proof} 
The proof of the estimates \eqref{241} and \eqref{242} can be found in Corollary 3.8 and Proposition 3.17 in  \cite{KPV3} respectively. To prove \eqref{242-1} we consider the analytic family of operators 
$$T_zu_0 = D_x^{-z/4}D_x^{1-z}S(t)u_0, \qquad \text{with}\;\; z\in \C,\; 0\leq \Re z\leq 1.$$ Now the estimate \eqref{242-1} follows by choosing $z=3/40$ in the Stein's theorem of analytic interpolation (see \cite{SW1}) between the smoothing estimate \eqref{221} and the maximal function estimate \eqref{231}.
\end{proof}

\begin{lemma}\label{lem-f}
Let $u_0\in L_x^2$, then for any $(\theta, \alpha)\in [0, 1]\times [0,\frac12]$, we have
\begin{equation}\label{eq-pq}
\|D_x^{\theta\alpha/2}S(t)u_0\|_{L_T^qL_x^p}\leq C\|u_0\|_{L_x^2},
\end{equation}
where $(q,p) = (\frac6{\theta(\alpha+1)}, \frac2{1-\theta})$.
\end{lemma}
\begin{proof}
See Lemma 2.4 in \cite{KPV0}.
\end{proof}

We state next the Leibniz's rule for fractional derivatives whose proof is
also given in  \cite{KPV3}, Theorem A.8.
\begin{lemma}\label{leibniz}
Let $\alpha \in (0, 1)$, $\alpha_1, \alpha_2 \in [0, \alpha]$,
 $\alpha_1+\alpha_2 = \alpha$. Let $p, p_1, p_2, q, q_1, q_2 \in (1, \infty)$
 be such that $\frac{1}{p} = \frac{1}{p_1}+\frac{1}{p_2}$,  $\frac{1}{q} =
 \frac{1}{q_1}+\frac{1}{q_2}$. Then
\begin{equation}\label{leib-1}
\|D_x^{\alpha}(fg) - fD_x^{\alpha}g-gD_x^{\alpha}f \|_{L_x^pL_T^q} \leq C
\|D_x^{\alpha_1}f\|_{L_x^{p_1}L_T^{q_1}} \|D_x^{\alpha_2}
g\|_{L_x^{p_2}L_T^{q_2}}.
\end{equation}
Moreover, for $\alpha_1 = 0 $ the value $q_1 = \infty$ is allowed.
\end{lemma}

\begin{definition}\label{def.2}
Let $1\leq p, \;q\leq \infty$, $-\frac14\leq\alpha\leq 1$. We say that a triple
$(p, q, \alpha)$ is an admissible triple if
\begin{equation}\label{admis}
\frac1p+\frac1{2q} =\frac14 \qquad \text{and}\qquad \alpha = \frac2q-\frac1p.
\end{equation}
\end{definition}

\begin{proposition}\label{kpv-7}
For any admissible triples $(p_j, q_j, \alpha_j)$, $j =1,2$, the following estimate holds
\begin{equation}\label{eq2.6}
\Big\|D_x^{\alpha_1}\int_0^{t}S(t-t')f(\cdot,t')dt'\Big\|_{L_x^{p_1}L_t^{q_1}} \leq C\|D_x^{-\alpha_2}f\|_{ L_x^{p_2'} L_t^{q_2'}},
\end{equation}
where $p_2', q_2'$ are the conjugate exponents of $p_2, q_2$.
\end{proposition}
\begin{proof}
For the proof we refer to Proposition 2.3 in \cite{KPV4}.
\end{proof}
The following results will be used to complete the contraction mapping argument.
\begin{lemma}\label{lema2.4}
Let $X_T$ and $Y_T$ be the spaces defined earlier and $S$ be the unitary group associated to the operator $\partial_x^3$, then we have

\begin{equation}\label{eq2.5.1}
\|S(t)u_0\|_{X_T}\leq C_0\|u_0\|_{H^1(\R)},
\end{equation}
\begin{equation}\label{eq2.5}
\Big\|\int_0^tS(t-t')f(t')dt'\Big\|_{X_T}\leq CT^{1/2}\|f\|_{Y_T}.
\end{equation}
\end{lemma}
\begin{proof}
The estimate \eqref{eq2.5.1} follows from the linear estimates in Lemmas \ref{lema2.1}, \ref{lema2.2} and \ref{lema2.3}. For the proof of the estimate \eqref{eq2.5}, we refer to our earlier work in \cite{CPS1}.
\end{proof}

\begin{lemma}\label{lema2.5}
The following estimate holds,
\begin{equation}\label{eq2.12}
\|\partial_x(u^{k+1})\|_{Y_T} \leq C\|u\|_{X_T}^{k+1}.
\end{equation}
\end{lemma}
\begin{proof}
The idea of the proof is similar to the one we used in \cite{CPS1} for the critical KdV equation. Using H\"older's inequality and the fact that  $H^1(\R)\hookrightarrow L^{\infty}(\R)$, we get
\begin{equation}\label{eq2.13}
\|\partial_x(u^{k+1})\|_{L_x^2L_T^2}\leq C\|u^{k-2}\|_{L_x^{\infty}L_T^{\infty}}\|u^2\partial_xu\|_{L_x^2L_T^2}\leq C\|u\|_{L_T^{\infty}H^1(\R)}^{k-2}\|u\|_{L_x^4L_T^{\infty}}^2\|\partial_xu\|_{L_x^{\infty}L_T^2}.
\end{equation}

Similarly
\begin{equation}\label{eq2.14}
\begin{split}
\|\partial_x^2(u^{k+1})\|_{L_x^2L_T^2}\!&\leq \!C\big[\|u^{k-1}(\partial_xu)^2\|_{L_x^2L_T^2} + 
\|u^k\partial_x^2u\|_{L_x^2L_T^2}\big]\\
&\leq \!C\big[\|u^{k-2}\|_{L_x^{\infty}L_T^{\infty}}\|u(\partial_xu)^2\|_{L_x^2L_T^2}+
\|u^{k-2}\|_{L_x^{\infty}L_T^{\infty}}\|u^2\partial_x^2u\|_{L_x^2L_T^2}\big]\\
&\leq \! C\|u\|_{L_T^{\infty}H^1(\R)}^{k-2}\big[\|u\|_{L_x^4L_T^{\infty}}\|\partial_xu\|_{L_x^{5}L_T^{10}}\|\partial_xu\|_{L_x^{20}L_T^{5/2}}\!+ \!\|u\|_{L_x^4L_T^{\infty}}^{k-2}\|\partial_x^2u\|_{L_x^{\infty}L_T^2}  \big].
\end{split}
\end{equation}

In view of definitions of $X_T$-norm and $Y_T$-norm, the estimates \eqref{eq2.13} and \eqref{eq2.14} yield the required result \eqref{eq2.12}.
\end{proof}

The following result from \cite{CS} will also be useful in our analysis.
\begin{lemma}\label{apn.1}
Let $T>0$, $1\leq p<q\leq \infty$ and $A, B \geq 0$. If $f\in L^q(0,T)$ satisfies
\begin{equation}\label{apn.2}
\|f\|_{L^q_{(0, t)}}\leq A +B\|f\|_{L^p_{(0, t)}},
\end{equation}
for all $t\in (0, T)$, then there exists a constant $K = K(B, p, q, T)$ such that
\begin{equation}\label{apn.3}
\|f\|_{L^q_{(0, T)}}\leq K A.
\end{equation}
\end{lemma}

%%%%%%%%%%%%%%%%%%%%%%%%%%%%%%%%%%%%%%%%%%%%%%%%%%%%%%%%%%%%%%%%%%%%%%%%%%%%%%%
\secao{Proof of the well-posedness results}\label{prmt}
%%%%%%%%%%%%%%%%%%%%%%%%%%%%%%%%%%%%%%%%%%%%%%%%%%%%%%%%%%%%%%%%%%%%%%%%%%%%%%%

We start this section by proving the well-posedness results for the IVP \eqref{ivp2} announced in Theorem \ref{th.1}. 

\begin{proof}[Proof of Theorem \ref{th.1}] 
For $a>0$, consider a ball in $X_T$ defined by
\begin{equation*}
\mathcal{B}_{a}^T = \{ u \in C([0, T] : X_T(\R)): \|u\|_{X_T} < a\}.
\end{equation*}

Our aim is to  show that, there exist $a>0$ and $T>0$,  such that the application $\Phi$
defined by 
\begin{equation}\label{in.eq-1}
\Phi(u):= S(t)\phi -\int_0^tS(t-t')\partial_x(u^{k+1})(t')dt',
\end{equation}
 maps $\mathcal{B}_{a}^T$ into $\mathcal{B}_{a}^T$ and is a contraction.

Using the estimates \eqref{eq2.5} and \eqref{eq2.12}, we obtain
\begin{equation}
\begin{split}
\|\Phi\|_{X_T}  &\leq C_0\|\phi\|_{H^1} + CT^{1/2}\|\partial_x(u^{k+1})\|_{Y_T}\\
&\leq C_0\|\phi\|_{H^1} + CT^{1/2}\|u\|_{X_T}^{k+1}.
\end{split}
\end{equation}

Hence, for $u\in\mathcal{B}_{a}^T$,
\begin{equation}\label{mx.8}
\|\Phi\|_{X_T}  \leq  C_0\|\phi\|_{H^1} + C T^{1/2} a^{k+1}.
\end{equation}

Now, choose  $a = 2C_0\|\phi\|_{H^1}$ and $T$ such that $ CT^{1/2}a^k<1/2$. With these choices
we get, from \eqref{mx.8},
\begin{equation*}
\|\Phi\|_{X_T} \leq \frac{a}{2} + \frac a2.
\end{equation*}

Therefore, $\Phi $ maps $\mathcal{B}_{a}^T$ into $\mathcal{B}_{a}^T$.

With the similar argument, one can prove that $\Phi $ is a contraction. The rest of the proof follows standard argument.
\end{proof}

\begin{remark}
From the choice of $a$ and $T$ in the proof of Theorem \ref{th.1} it is clear that the local existence time
is given by
\begin{equation}\label{eq.b1}
T \leq C\|\phi\|_{H^1(\R)}^{-2k}.
\end{equation}
Moreover, we have the following bound,
\begin{equation}\label{eq2.b2}
\|u\|_{X_T} \leq C\|\phi\|_{H^1(\R)}.
\end{equation}
\end{remark}

In what follows, we sketch a proof for the local well-posedness result for the IVP \eqref{ivp1.1}.
\begin{proof}[Proof of Theorem \ref{th.2}]
As in the proof of Theorem \ref{th.1}, this theorem will also be proved  by considering  the
integral equation associated to the IVP \eqref{ivp1.1}, 
\begin{equation}\label{os-1}
u(t)= S(t)\phi - \int_{0}^{t}S(t-t')g(\omega(t'+t_0))\partial_x(u^{k+1})(t')\,dt',
\end{equation}
and using the contraction mapping principle.

First of all, notice that the periodic function $g$ is bounded,  say $\|g\|_{L_t^{\infty}} \leq A$, for some positive constant $A$. Since the norms involved in the space $Y$ permit us to take out $\|g\|_{L_t^{\infty}}$-norm as a coefficient, the proof of this theorem follows exactly  the same argument as in the proof of Theorem \ref{th.1}. Moreover, as the initial data $\phi$ is the same, the choice of the radius $a$ of the ball is exactly the same. However, to complete the contraction mapping argument, we must select $T>0$ such that 
$
C\|g\|_{L_t^{\infty}}T^{1/2}a^4<\frac12,
$
which implies that the existence  $T$ is given by
\begin{equation}\label{tempex}
T=T(\|g\|_{L_t^{\infty}},\|\phi\|_{H^1(\R)})=\frac{C}{\|g\|_{L_t^{\infty}}^2 \|\phi\|_{H^1(\R)}^{2k}}.
\end{equation} Furthermore, in this case too, from the proof, one can get
\begin{equation}\label{pr.1}
\|u\|_{X_T}\leq C  \|\phi\|_{H^1(\R)}.
\end{equation}
\end{proof}

In sequel, we present some results that play a central role in the proof of the main theorems of this work. We begin with the following lemma whose proof can be found in \cite{CPS1}.
\begin{lemma}\label{lem.1}
Let $X_T$ and $Y_T$ be spaces as defined in \eqref{n.11} and \eqref{n.21}. Let $f\in Y_T$, then we have the following convergence
\begin{equation}\label{cv.1}
\int_0^tg(\omega (t'+t_0))S(t-t')f(t')dt' \to m(g) \int_0^tS(t-t')f(t')dt',
\end{equation}
whenever $|\omega|\to \infty$, in the $X_T$-norm.
\end{lemma}

With the similar argument as in the case of the  critical KdV equation (see \cite{CPS1}), we have the following convergence result.
\begin{lemma}\label{lem.2}
Let the initial data $\phi \in H^1(\R)$. Let $u_{\omega, t_0}$ be the maximal solution of the IVP \eqref{ivp1}. Suppose $U$ be the maximal solution of the IVP \eqref{ivp3} defined in $[0, S_{max})$. Let $0<T<S_{max}$ and let $u_{\omega, t_0}$ exists in $[0, T]$ for $|\omega|$ large and that
\begin{equation}\label{x.1}
\limsup_{|\omega|\to \infty}\sup_{t_0\in \R}\|u_{\omega, t_0}\|_{L_T^{\infty}H^1(\R)}<\infty,
\end{equation}
and
\begin{equation}\label{x.1.1}
\limsup_{|\omega|\to \infty}\sup_{t_0\in \R}\|u_{\omega, t_0}\|_{L_x^4L_T^{\infty}}<\infty.
\end{equation}
Then, for all $t \in [0, T]$,
\begin{equation}\label{x.2}
\sup_{t_0\in \R}\|u_{\omega, t_0}-U\|_{X_T}\to 0 , \quad {\text{as}}\;\; |\omega|\to \infty.
\end{equation}
In particular, $u_{\omega, t_0}\to U$ as $|\omega|\to \infty$, in $H^1(\R)$.
\end{lemma}
\begin{proof} Since $u_{\omega, t_0}$ and $U$ have the same initial data $\phi$, from Duhamel's formula, we have
\begin{equation}\label{x.3}
\begin{split}
u_{\omega, t_0} -U & = \int_0^t g(\omega(t'+t_0))S(t-t')\partial_x(u_{\omega, t_0}^{k+1}) dt' - m(g)\int_0^t S(t-t')\partial_x(U^{k+1})dt'\\
& = \int_0^t g(\omega(t'+t_0))S(t-t')\partial_x(u_{\omega, t_0}^{k+1}-U^{k+1}) dt'\\
&\qquad +\int_0^t[g(\omega(t'+t_0)) -m(g)]S(t-t')\partial_x(U^{k+1})dt'\\
&=: I_1+I_2.
\end{split}
\end{equation}

We note that
\begin{equation}\label{mm.1}
|u^{k+1}-v^{k+1}|\leq C(|u|^k+|v|^k)|u-v|
\end{equation}
 and
 \begin{equation}\label{mm.2}
 |\partial_x(u^{k+1}-v^{k+1})|\leq C\big[(|u|^k+|v|^k)|\partial_x(u-v)| + (|\partial_xu|+|\partial_xv|)(|u|^{k-1}+|v|^{k-1})|u-v|\big].
 \end{equation}

Let $\|g\|_{L_T^{\infty}}\leq A$. Use of  \eqref{eq2.2}, \eqref{mm.1}, H\"older's inequality and the assumptions \eqref{x.1} and \eqref{x.1.1}, yield
\begin{equation}\label{x.4}
\begin{split}
\|I_1\|_{L_T^{\infty}L_x^2} &\leq C\|g\|_{L_T^{\infty}} \|u_{\omega, t_0}^{k+1} -U^{k+1}\|_{L_x^{1}L_T^2}\\
& \leq CA\|u_{\omega, t_0}^k(u_{\omega, t_0}-U)\|_{L_x^{1}L_T^2} +\|U^k(u_{\omega, t_0} -U)\|_{L_x^{1}L_T^2}\\
&\leq CA \|u_{\omega, t_0}^k\|_{L_x^2L_T^{\infty}}\|u_{\omega, t_0}-U\|_{L_x^{2}L_T^2} +\|U^k\|_{L_x^2L_T^{\infty}}\|u_{\omega, t_0} -U\|_{L_x^{2}L_T^2}\\
&\leq CA \Big[\|u_{\omega, t_0}^{k-2}\|_{L_x^{\infty}L_T^{\infty}}\|u_{\omega, t_0}^{2}\|_{L_x^2L_T^{\infty}} +\|U^{k-2}\|_{L_x^{\infty}L_T^{\infty}}\|U^2\|_{L_x^2L_T^{\infty}}\Big]\|u_{\omega, t_0} -U\|_{L_T^2L_x^{2}}\\
&\leq CA \Big[\|u_{\omega, t_0}\|_{L_T^{\infty}H^1(\R)}^{k-2}\|u_{\omega, t_0}\|_{L_x^4L_T^{\infty}}^2 +\|U\|_{L_T^{\infty}H^1(\R)}^{k-2}\|U\|_{L_x^4L_T^{\infty}}^2\Big]\|u_{\omega, t_0} -U\|_{L_T^2L_x^{2}}\\
&\leq CA\|u_{\omega, t_0} -U\|_{L_T^2L_x^{2}}.
\end{split}
\end{equation}

Again, using \eqref{eq2.2} and \eqref{mm.2},  one can obtain
\begin{equation}\label{msx.1}
\begin{split}
\|\partial_xI_1\|_{L_T^{\infty}L_x^2}& \leq CA \|\partial_x(u_{\omega, t_0}^{k+1} -U^{k+1})\|_{L_x^{1}L_T^2}\\
&\leq CA\Big[\|(|u_{\omega, t_0}|^k+|U|^k)\partial_x(u_{\omega, t_0}-U)\|_{L_x^{1}L_T^2}\\
&\qquad\qquad\quad + \|(|\partial_xu_{\omega, t_0}|+|\partial_xU|)(|u_{\omega, t_0}|^{k-1}+|U|^{k-1})(u_{\omega, t_0}-U)\|_{L_x^{1}L_T^2}\Big]\\
&=: CA[J_{1}+J_{2}].
\end{split}
\end{equation}

With the same argument as in \eqref{x.4}
\begin{equation}\label{msx.3}
J_1\leq C\|\partial_x(u_{\omega, t_0} -U)\|_{L_T^2L_x^{2}}.
\end{equation}

Now we move to estimate the first term, $\|u_{\omega, t_0}^{k-1}\partial_xu_{\omega, t_0}(u_{\omega, t_0}-U)\|_{L_x^{1}L_T^2}$ in $J_2$, the estimates for the other terms are similar. We have,
\begin{equation}\label{msx.7}
\begin{split}
\|u_{\omega, t_0}^{k-1}\partial_xu_{\omega, t_0}(u_{\omega, t_0}-U)\|_{L_x^{1}L_T^2}
&\leq C\|u_{\omega, t_0}^2\|_{L_x^2L_T^{\infty}}\|u_{\omega, t_0}^{k-3}\partial_xu_{\omega, t_0}(u_{\omega, t_0}-U)\|_{L_x^2L_T^2}\\
%&\leq C\|u_{\omega, t_0}\|_{L_x^4L_T^{\infty}}^2\|u_{\omega, t_0}
%\partial_xu_{\omega, t_0}(u_{\omega, t_0}-U)\|_{L_T^2L_x^2}\\
&\leq C\|u_{\omega, t_0}\|_{L_x^4L_T^{\infty}}^2\|u_{\omega, t_0}^{k-3}\|_{L_T^{\infty}L_x^{\infty}} \|\partial_xu_{\omega, t_0}\|_{L_T^{\infty}L_x^2}\|(u_{\omega, t_0}-U)\|_{L_T^2L_x^{\infty}}\\
&\leq C \|u_{\omega, t_0}\|_{L_x^4L_T^{\infty}}^2\|u_{\omega, t_0}\|_{L_T^{\infty}H^1(\R)}^{k-2} \|(u_{\omega, t_0}-U)\|_{L_T^2H^1(\R)}\\
&\leq C \|(u_{\omega, t_0}-U)\|_{L_T^2H^1(\R)}.
\end{split}
\end{equation}

Inserting \eqref{msx.3} and \eqref{msx.7} in \eqref{msx.1}, we get
\begin{equation}\label{msx.8}
\|\partial_xI_1\|_{L_T^{\infty}L_x^2}\leq CA\|(u_{\omega, t_0}-U)\|_{L_T^2H^1(\R)}.
\end{equation}

Combining \eqref{x.4} and \eqref{msx.8}, we obtain
\begin{equation}\label{msx.9}
\|I_1\|_{L_T^{\infty}H^1(\R)}\leq CA\|(u_{\omega, t_0}-U)\|_{L_T^2H^1(\R)}.
\end{equation}

From Lemma \ref{lem.1}, we have
\begin{equation}\label{x.5}
\|I_2\|_{L_T^{\infty}H^1(\R)}\leq C_{\omega}\to 0, \quad \mbox{as} \;\; |\omega|\to \infty.
\end{equation}

Therefore, we have
\begin{equation}\label{msx.10}
\|u_{\omega, t_0}-U\|_{L_T^{\infty}H^1(\R)}\leq CA\|(u_{\omega, t_0}-U)\|_{L_T^2H^1(\R)} +C_{\omega}.
\end{equation}

Applying Lemma \ref{apn.1} in \eqref{msx.10}, we get
\begin{equation}\label{msx.4}
\|u_{\omega, t_0} -U\|_{L_T^{\infty}H^1(\R)}\leq KC_{\omega} \to 0, \qquad as \; |\omega|\to \infty.
\end{equation}

From \eqref{msx.10} and \eqref{msx.4}, it is easy to conclude that
\begin{equation}\label{eq3.28}
\|(u_{\omega, t_0}-U)\|_{L_T^2H^1(\R)} \to 0,\quad as\; |\omega|\to \infty.
\end{equation}

Now, we move to estimate the other norms involved in the definition of $X_T$. Let, 
$$\mathfrak{L_1}:= \|\partial_x(u_{\omega, t_0} -U)\|_{L_x^{\infty}L_T^{2}} +\|\partial_x^2(u_{\omega, t_0} -U)\|_{L_x^{\infty}L_T^{2}}+\|u_{\omega, t_0} -U\|_{L_x^{5}L_T^{10}}+\|D_x(u_{\omega, t_0} -U)\|_{L_x^{5}L_T^{10}}$$
and
$$\mathfrak{L_2}:=\|\partial_x(u_{\omega, t_0} -U)\|_{L_x^{20}L_T^{5/2}}+\|u_{\omega, t_0} -U\|_{L_x^{4}L_T^{\infty}}.$$

Use of \eqref{eq2.2}, \eqref{eq2.3}, the estimate \eqref{eq2.6} from Proposition \ref{kpv-7} with admissible triples $(p_1,q_1, \alpha_1)=(5,10,0)$, and $(p_2,q_2, \alpha_2)=(\infty,2,1)$  in \eqref{x.3}, yields
\begin{equation}\label{cps11}
\mathfrak{L_1} \le CA \|\partial_x(u_{\omega, t_0}^{k+1} -U^{k+1})\|_{L_x^{1}L_T^2} +CA \|u_{\omega, t_0}^{k+1} -U^{k+1}\|_{L_x^{1}L_T^2} +\|I_2\|_{X_T}.
\end{equation}
Therefore,  with the same argument as in \eqref{x.4}-\eqref{msx.8}, we can obtain
\begin{equation}\label{eq3.30}
\mathfrak{L_1} \le  CA \|u_{\omega, t_0} -U\|_{L_T^2 H^{1}} +C_{\omega}. 
\end{equation}

Hence, using Lemma \ref{lem.2} and \eqref{eq3.28} we get from \eqref{eq3.30} that
\begin{align}\label{cps10}
\mathfrak{L_1} \stackrel{|\omega| \to \infty}{\rightarrow} 0.
\end{align}

Finally, to estimate  $\mathfrak{L_2}$ we use Proposition \ref{kpv-7} with admissible triples $(p_1,q_1, \alpha_1)=(20,5/2,3/4)$ and $(p_2,q_2, \alpha_2)=(20/3,5,1/4)$, to get
\begin{equation}\label{cps13}
\Big\|\partial_x\int_0^{t}S(t-t')f(\cdot,t')dt'\Big\|_{L_x^{20}L_T^{5/2}} \leq C\|f\|_{L_x^{20/17} L_T^{5/4}},
\end{equation}
and with admissible triples $(p_1,q_1, \alpha_1)=(4,\infty,-1/4)$, and $(p_2,q_2, \alpha_2)=(20/3,5,1/4)$, to have
\begin{equation}\label{cps14}
\Big\|\int_0^{t}S(t-t')f(\cdot,t')dt'\Big\|_{L_x^{4}L_T^\infty} \leq C\|f\|_{L_x^{20/17} L_T^{5/4}}.
\end{equation}

Using \eqref{cps13}, \eqref{cps14}, and the definition of $X_T$, we get from \eqref{x.3} that
\begin{equation}
\label{cps15}
\mathfrak{L_2} \le C A \|\partial_x(u_{\omega, t_0}^{k+1} -U^{k+1})\|_{L_x^{20/17}L_T^{5/4}}+\|I_2\|_{X_T}
\end{equation}

Using  \eqref{mm.2},  we can obtain
\begin{equation}\label{eq3.35}
\begin{split}
\|\partial_x(u_{\omega, t_0}^{k+1} -U^{k+1})\|_{L_x^{20/17}L_T^{5/4}}
&\leq C\Big[\|(|u_{\omega, t_0}|^k+|U|^k)\partial_x(u_{\omega, t_0}-U)\|_{L_x^{20/17}L_T^{5/4}}\\
&\qquad+ \|(|\partial_xu_{\omega, t_0}|+|\partial_xU|)(|u_{\omega, t_0}|^{k-1}+|U|^{k-1})(u_{\omega, t_0}-U)\|_{L_x^{20/17}L_T^{5/4}}\Big]\\
&=: C[\tilde{J}_{1}+\tilde{J}_{2}].
\end{split}
\end{equation}
 H\"older's inequality, the fact that $20/13>10/7$, Sobolev immersion and the assumption \eqref{x.1}, imply that
 \begin{equation}\label{eq3.36}
 \begin{split}
 \tilde{J}_1&\leq C \|\partial_x(u_{\omega, t_0} -U)\|_{L_x^{5}L_T^{10}}\{\|u_{\omega, t_0}^{k}\|_{L_x^{20/13}L_T^{10/7}}+\|U^{k}\|_{L_x^{20/13}L_T^{10/7}}\}\\
 &\leq C \|\partial_x(u_{\omega, t_0} -U)\|_{L_x^{5}L_T^{10}}\{\|u_{\omega,t_0}^{k}\|_{L_T^{10/7}L_x^{20/13}} +\|U^{k}\|_{L_T^{10/7}L_x^{20/13}}\}\\
 &\leq C \|\partial_x(u_{\omega, t_0} -U)\|_{L_x^{5}L_T^{10}} \, T^{7/10}\{\|u_{\omega, t_0}\|_{L_T^\infty H^1}^{k}+\|U\|_{L_T^\infty H^1}^{k}\}\\
 &\leq C\, T^{7/10}\|\partial_x(u_{\omega, t_0} -U)\|_{L_x^{5}L_T^{10}}.
 \end{split}
 \end{equation}
 
 An in \eqref{msx.1}, we give details in estimating the first term, 
 $\|u_{\omega, t_0}^{k-1}\partial_xu_{\omega, t_0}(u_{\omega, t_0}-U)\|_{L_x^{20/17}L_T^{5/4}}$ in $\tilde{J}_2$, the estimates for the other terms are similar. Here too, H\"older's inequality, the fact that $20/3>5$, Sobolev immersion and the assumption \eqref{x.1}, yield
\begin{equation}\label{eq3.37}
\begin{split}
\|u_{\omega, t_0}^{k-1}\partial_xu_{\omega, t_0}(u_{\omega, t_0}-U)\|_{L_x^{20/17}L_T^{5/4}} 
&\leq C \|u_{\omega, t_0}^{k-1}\|_{L_x^{20/3}L_T^{5}}\|\partial_x u_{\omega, t_0}\|_{L_x^{2}L_T^{2}}\|u_{\omega, t_0} -U\|_{L_x^{5}L_T^{10}}\\
&\leq C 
\|u_{\omega, t_0}^{k-1}\|_{L_T^{5}L_x^{20/3}}\|\partial_x u_{\omega, t_0}\|_{L_T^{2}L_x^{2}}\|u_{\omega, t_0} -U\|_{L_x^{5}L_T^{10}}\\
&\leq C \,T^{7/10}\|u_{\omega, t_0}\|_{L_T^\infty H^1}^{k}\|u_{\omega, t_0} -U\|_{L_x^{5}L_T^{10}}\\
&\leq C\, T^{7/10}\|u_{\omega, t_0} -U\|_{L_x^{5}L_T^{10}}.
\end{split}
\end{equation}
In view of \eqref{eq3.35}, \eqref{eq3.36} and \eqref{eq3.37}, we get from \eqref{cps15} that
\begin{equation}\label{cps17}
\mathfrak{L_2}\leq CA\,T^{7/10}\{\|\partial_x(u_{\omega, t_0} -U)\|_{L_x^{5}L_T^{10}}+\|u_{\omega, t_0} -U\|_{L_x^{5}L_T^{10}}\}+C_{\omega}.
\end{equation}
Therefore, Lemma \ref{lem.2} and \eqref{cps10}, imply
\begin{equation}\label{cps16}
\mathfrak{L_2}\stackrel{|\omega| \to \infty}{\rightarrow} 0.
\end{equation}

Now, the proof of the Lemma follows by combining \eqref{msx.4}, \eqref{cps10} and \eqref{cps16}.
\end{proof}

In what follows, as we did in our earlier work \cite{CPS1}, we consider the supercritical KdV equation with more general time dependent coefficient on the nonlinearity. Given $h \in L^{\infty}$ we consider
\begin{equation}\label{xav1}
  \begin{cases}
     u_t+u_{xxx}+h(t) \partial_x(u^{k+1})=0, \quad x, \, t\in
     \mathbb{R},\; k\geq 5\\
     u(x, 0)=\phi(x).
   \end{cases}
\end{equation}

The results for the IVP \eqref{xav1} and their proofs that we are going to present here are quite similar to the ones we have for the critical KdV equation in \cite{CPS1}. For the sake of clarity, we reproduce them here.
\begin{proposition}\label{xavprop1}
Given any $A>0$, there exist $\epsilon =\epsilon(A)$ and $B>0$ such that if $\|h\|_{L^{\infty}}\le A$ and if $\phi \in H^{1}(\R) $ satisfies
\begin{equation}\label{eq3.31}
\|S(t) \phi\|_{L_x^5L_t^{10}} \le \epsilon,
\end{equation}
then the corresponding solution $u$ of \eqref{xav1} is global and satisfies
\begin{equation}\label{eq3.32}
\|u \|_{L_x^5L_t^{10}} \le 2 \, \|S(t) \phi \|_{L_x^5L_t^{10}},
\end{equation}
\begin{equation}\label{eq3.33}
%|\!|\!|u|\!|\!|_{1}
\|u\|_{X_t}
\le B \| \phi\|_{H^{1}(\mathbb{R})}.
\end{equation}
Conversely, if the solution $u$ of \eqref{xav1} is global and satisfies
\begin{equation}\label{eq3.34}
\|u \|_{L_x^5L_t^{10}} \le \epsilon,
\end{equation}
then
\begin{equation}\label{xav12}
\|S(t) \phi\|_{L_x^5L_t^{10}}\le 2\|u \|_{L_x^5L_t^{10}}.
\end{equation}
\end{proposition}
\begin{proof}
Since $\|h\|_{L_t^{\infty}}\leq A$, as in Theorem \ref{th.2} we can prove the local well-posedness for the IVP \eqref{xav1} in $H^1(\R)$ with time of existence  $T=T(\|\phi\|_{H^1(\R)}, \|h\|_{L^{\infty}})$. Let $u\in C([0, T_{max}); H^1(\R))$ be the maximal solution of the IVP \eqref{xav1}. For $0\leq t <T_{max}$, we have that
\begin{equation}\label{xav14}
u(t)=S(t) \phi+w(t),
\end{equation} where
$$
w(t)= -\int_{0}^{t}S(t-t')h(t')\partial_x(u^{k+1})(t')\,dt'.
$$

Using \eqref{eq2.6} from Proposition \ref{kpv-7} for admissible triples $(5, 10, 0)$ and $(\infty, 2, 1)$, we obtain
\begin{equation}\label{xav15}
\begin{split}
\|w\|_{L_x^5L_T^{10}} &\leq   C A \|u^{k+1}\|_{L_x^1L_T^{2}}\leq CA\|u^{k-4}\|_{L_x^{\infty}L_T^{\infty}}\|u^5\|_{L_x^1L_T^{2}}\\
&\leq CA\|u\|_{L_T^{\infty}H^1}^{k-4}\|u\|_{L_x^5L_T^{10}}^5\leq C A \|u\|_{L_x^5L_T^{10}}^5.
\end{split}
\end{equation}

From \eqref{xav14} and \eqref{xav15} it follows that
\begin{equation}\label{xav16}
|\,\|u \|_{L_x^5L_T^{10}}-\|S(t) \phi \|_{L_x^5L_T^{10}}|\leq  C A \|u\|_{L_x^5L_T^{10}}^5.
\end{equation}

Thus, for all $T \in (0, T_{max})$ one has
\begin{equation}\label{eq3.41}
\|u\|_{L_x^5L_T^{10}}\le \epsilon + C A \|u\|_{L_x^5L_T^{10}}^5.
\end{equation}

Choose $\epsilon=\epsilon(A)$ such that
\begin{equation}
\label{eq3.42}
C A (2\epsilon)^4 <1/2,
\end{equation}
and suppose that the estimate \eqref{eq3.31} holds.  As the norm is  continuous on $T$ and vanishes at $T=0$,  using continuity argument, the estimate \eqref{eq3.41} and the  choice of $\epsilon$ in \eqref{eq3.42}, imply that
\begin{equation}\label{eq3.43}
\|u \|_{L_x^5L^{10}_{T_{max}}}\le 2 \epsilon.
\end{equation}
Moreover, from \eqref{xav16}
\begin{equation}\label{eq39}
\begin{split}
\|u\|_{L_x^5L^{10}_{T_{max}}} & \le  \|S(t) \phi\|_{L_x^5L^{10}_{T_{max}}} + C A \|u\|_{L_x^5L^{10}_{T_{max}}}^5\\
&\leq\|S(t) \phi\|_{L_x^5L^{10}_{T_{max}}} +C A (2\epsilon)^4\|u\|_{L_x^5L^{10}_{T_{max}}}.
\end{split}
\end{equation}

Therefore,  with the choice of $\epsilon$ satisfying \eqref{eq3.42},  the estimate \eqref{eq39} yields
\begin{equation}\label{eq3.45}
\|u\|_{L_x^5L_{T_{max}}^{10}}\leq 2 \|S(t) \phi\|_{L_x^5L_{T_{max}}^{10}}.
\end{equation}

In what follows, we will show that $T_{max} = \infty$. The inequalities \eqref{eq2.2}, \eqref{eq2.3}, \eqref{eq2.6} with admissible triples $(5,10,0)$ and $(\infty, 2, 1)$, and H\"older's inequality imply
\begin{equation}\label{0xav27}
\|w\|_{L_T^{\infty}H^1}+ \|\partial_xw\|_{L_x^{\infty}L_T^2} +
\|\partial_x^2w\|_{L_x^{\infty}L_T^2}
 + \|w\|_{L_x^{5}L_T^{10}}+  \|\partial_xf\|_{L_x^{5}L_T^{10}}
 \le C A\|u\|_{L_x^5L_{T}^{10}}^4\| u\|_{X_T}.
\end{equation}

Now using \eqref{cps13}, \eqref{cps14} and H\"older's inequality, we have
\begin{align}\label{1xav27}
\|\partial_x w\|_{L_x^{20}L_T^{5/2}}+\| w\|_{L_x^{4}L_T^{\infty}} \le & C A \|\partial_x (u^{k+1})\|_{L_x^{20/17}L_T^{5/4}}\nonumber \\
\le & CA \|u^k\|_{L_x^{5/4}L_T^{5/2}}\|\partial_x u\|_{L_x^{20}L_T^{5/2}}\nonumber\\
\le & CA \|u^{k-4}\|_{L_x^{\infty}L_T^{\infty}}\|u^4\|_{L_x^{5/4}L_T^{5/2}}\|\partial_x u\|_{L_x^{20}L_T^{5/2}}\nonumber\\
\le & CA \|u\|_{L_T^{\infty}H^1}^{k-4}\|u\|_{L_x^{5}L_T^{10}}^4\|\partial_x u\|_{L_x^{20}L_T^{5/2}}\nonumber\\
\le & CA \|u\|_{L_x^{5}L_T^{10}}^4\|\partial_x u\|_{L_x^{20}L_T^{5/2}}.
\end{align}

Combining  \eqref{0xav27} and \eqref{1xav27}, we obtain
\begin{align}\label{xav27}
\|w\|_{X_T} \le C A \|u\|_{L_x^5L_{T}^{10}}^4 \|u\|_{X_T}.
\end{align}
This estimate with  \eqref{eq3.42} and \eqref{eq3.43} gives
\begin{equation}\label{eq3.47}
\|w\|_{X_T} \le C A(2\epsilon)^4 \|u\|_{X_T} <\frac12 \|u\|_{X_T}.
\end{equation}

Using \eqref{xav14} we obtain
\begin{equation}
\|u\|_{X_T} \le  \|S(t) \phi\|_{X_T}+  \|w\|_{X_T} \le
C\|\phi\|_{H^1(\mathbb{R})}+\frac12 \|u\|_{X_T},
\end{equation}
for all $T \in (0,T_{max})$. Therefore, we have
\begin{equation}\label{ap-21}
\|u\|_{X_{T_{max}}} \le 2\, C \| \phi\|_{H^1(\R)}.
\end{equation}
Hence, from the definition of $\|u\|_{X_{T_{max}}}$, we have that
 \begin{equation}\label{ap-22}
 \|u\|_{ L^\infty_{T_{max}}H^1(\mathbb{R})} \le C \|u(0)\|_{H^1(\R)}.
\end{equation}

Now, combining the local existence from Theorem \ref{th.2} and the estimate \eqref{ap-22}, the blow-up alternative implies that $T_{max}=\infty$. Finally, the estimates \eqref{eq3.45} and \eqref{ap-21} yield \eqref{eq3.32} and \eqref{eq3.33} respectively with $B=2C$.

Conversely, let $T_{max}=\infty$ and   \eqref{eq3.34} holds. With the similar argument as in \eqref{xav16}, we can get
\begin{equation}\label{xav51}
|\,\|u \|_{L_x^5L_t^{10}}-\|S(t) \phi \|_{L_x^5L_t^{10}}|\leq  C A \|u\|_{L_x^5L_t^{10}}^5.
\end{equation}

Thus, from \eqref{xav51} in view of \eqref{eq3.34} and \eqref{eq3.42}, one has
\begin{equation}\label{xav52}
\|S(t)\phi\|_{L_x^5L_t^{10}}\le \|u\|_{L_x^5L_t^{10}} + C A \epsilon^4\|u\|_{L_x^5L_t^{10}} \leq 2\|u\|_{L_x^5L_t^{10}}.
\end{equation}
\end{proof}

\begin{corollary}\label{xavcor}
Let $h \in L^{\infty}(\mathbb{R})$ satisfy $\|h\|_{L^{\infty}} \le A $ and  $\epsilon$ and $B$ be as in Proposition \ref{xavprop1}. Given $\phi\in H^{1}(\mathbb{R})$, let $u$ be the solution of the IVP \eqref{xav1} defined on the maximal interval $[0,T_{max})$. If there exists $T \in (0,T_{max})$ such that
$$
\|S(t) u(T)\|_{L_x^5L_t^{10}} \le \epsilon,
$$
then  the solution $u$ is global. Moreover
$$
\|u\|_{L_x^5L_{(T, \infty)}^{10}}\le 2\epsilon, \quad \textrm{and} \quad \|u\|_{X_{(T,\infty)}} \le B \|u(T)\|_{H^{1}(\R)}.
$$
\end{corollary}
\begin{proof}
The proof follows by using a standard extension argument. For details we refer to the  proof of Corollary 2.4 in \cite{CS}.
\end{proof}

%%%%%%%%%%%%%%%%%%%%%%%%%%%%%%%%%%%%%%%%%%%%%%%%%%%%%%%%%%%%%%%%%%%%%%%%
%%%%%%%%%%%%%%%%%%%%%%%%%%%%%%%%%%%%%%%%%%%%%%%%%%%%%%%%%%%%%%%%%%%%%%%%
\secao{Proof of the main results}\label{prmt-2}
%%%%%%%%%%%%%%%%%%%%%%%%%%%%%%%%%%%%%%%%%%%%%%%%%%%%%%%%%%%%%%%%%%%%%%%
%%%%%%%%%%%%%%%%%%%%%%%%%%%%%%%%%%%%%%%%%%%%%%%%%%%%%%%%%%%%%%%%%%%%%%%%%

The argument in the proof of the main results, Theorem \ref{main.th1} and Theorem \ref{main.th2}, is quite similar to the one used in the case of the critical KdV equation \cite{CPS1}. As mentioned earlier,   Lemma \ref{lem.2} and the local existence Theorem \ref{th.2} are  used in the proof of Theorem \ref{main.th1}. While, Proposition \ref{xavprop1} and Theorem \ref{main.th1} are crucial in the proof of Theorem \ref{main.th2}. He we adapt the techniques used in \cite{CPS1} and \cite{CS} to complete the proofs.

\begin{proof}[Proof of Theorem \ref{main.th1}] Let $A=\|g\|_{L^\infty}$, $T \in (0, S_{\max})$ fixed  and set
\begin{equation}\label{eq4.1}
M_0=2 \sup_{t \in [0,T]}\|U(t)\|_{H^1(\R)}.
\end{equation}

In particular, for $t=0$, \eqref{eq4.1} gives $\|\phi\|_{H^1(\R)} \le M_0/2$. From  Theorem \ref{th.2}, we have that for all $\omega, t_0\in \R$, $u_{\omega, t_0}$ exists on $[0,\delta]$. Using \eqref{tempex} we have that the existence time $\delta$, is given by
\begin{equation}\label{tempex1}
\delta=\dfrac{C}{A^2M_0^8}.
\end{equation}

Moreover, from \eqref{pr.1}
\begin{equation}\label{eq4.2}
\limsup_{|w| \to \infty} \sup_{t_0 \in \mathbb{R}}\|u_{\omega, t_0}\|_{L_{\delta}^\infty H^1(\R)} \le C \|\phi\|_{H^1(\R)}
\end{equation}
and
\begin{equation}\label{eq4.0}
\limsup_{|w| \to \infty} \sup_{t_0 \in \mathbb{R}}\|u_{\omega, t_0}\|_{L_x^4L_{\delta}^\infty H^1(\R)} \le C \|\phi\|_{H^1(\R)}.
\end{equation}
From Lemma \ref{lem.2}, we have that $\sup_{t_0 \in \mathbb{R}}\|u_{\omega, t_0}- U\|_{X_T}\stackrel{|w| \to \infty}{\rightarrow} 0$, in particular
\begin{equation}\label{eq4.4}
\sup_{t_0 \in \mathbb{R}}\|u_{\omega, t_0}(\delta)- U(\delta)\|_{H^1(\mathbb{R})}\stackrel{|w| \to \infty}{\rightarrow} 0.
\end{equation}
Combining \eqref{eq4.1} and \eqref{eq4.4}, for $|w|$ sufficiently large, we deduce that
\begin{equation}\label{eq4.5}
\sup_{t_0 \in \mathbb{R}}\|u_{\omega, t_0}(\delta)\|_{H^1(\R)} \le M_0.
\end{equation}

 We suppose $\delta \leq T$, otherwise we are done. Using Theorem \ref{th.2} we can extend the solution $u_{\omega, t_0}$ (as in the proof of Corollary \ref{xavcor}) on the interval $[0, 2 \delta]$, with $\|\tilde{u}_{\omega, t_0}\|_{L_t^\infty(0,\delta) H^1(\R)} \le C \|\tilde{u}_{\omega, t_0}(0)\|_{H^1(\R)}$, where $\tilde{u}_{\omega, t_0}(t)=u_{\omega, t_0}(t+\delta)$ i.e., $\|u_{\omega, t_0}\|_{L_t^\infty(\delta, 2\delta) H^1(\R)} \le C \|u_{\omega, t_0}(\delta)\|_{H^1(\R)} \le C^2\|\phi\|_{H^1(\R)}$. Therefore, \eqref{eq4.2} gives
\begin{equation}
\limsup_{|w| \to \infty} \sup_{t_0 \in \mathbb{R}}\|u_{\omega, t_0}\|_{L_t^\infty(0, 2\delta) H^1(\R)} \le C(1+C)\|\phi\|_{H^1(\R)}.
\end{equation}
Similarly, from \eqref{eq4.0},
\begin{equation}
\limsup_{|w| \to \infty} \sup_{t_0 \in \mathbb{R}}\|u_{\omega, t_0}\|_{L_x^4L_{2\delta}^\infty H^1(\R)} \le C(1+C)\|\phi\|_{H^1(\R)}.
\end{equation}
So, we can again apply the Lemma \ref{lem.2}. Iterating this argument at a finite number of times with the same time of existence in each iteration, we see that
\begin{equation*}
\limsup_{|w| \to \infty} \sup_{t_0 \in \mathbb{R}}\|u_{\omega, t_0}\|_{L_T^\infty H^1(\R)} \le C\|\phi\|_{H^1(\R)}
\end{equation*}
and
\begin{equation*}
\limsup_{|w| \to \infty} \sup_{t_0 \in \mathbb{R}}\|u_{\omega, t_0}\|_{L_x^4L_T^\infty} \le C\|\phi\|_{H^1(\R)}.
\end{equation*}
The result is therefore a consequence of Lemma \ref{lem.2}.
\end{proof}

%%%%%%%%%%%%%%%%%%%%%%%%%%%%%%%%%%%%%%%%%%%%%%%%%%%%%%%%%%%%%%%%%%%%%%%%%%%

\begin{proof}[Proof of Theorem \ref{main.th2}]
Let $\epsilon \in (0,\epsilon(A))$, where $\epsilon(A)$ is as in Proposition \ref{xavprop1}. If $T$ is sufficiently large, from \eqref{eq1.20}, we have that
\begin{equation}\label{xav33}
\|U\|_{L_x^5L_{(T, \infty)}^{10}}\le \frac{\epsilon}{4}.
\end{equation}
Applying Proposition \ref{xavprop1} to the global solution $\tilde{U}(t)=U(t+T)$, the inequality \eqref{xav12} gives
\begin{equation}
\|S(t) U(T)\|_{L_x^5L_t^{10}}=\|S(t) \tilde{U}(0)\|_{L_x^5L_t^{10}} \le 2\|\tilde{U}\|_{L_x^5L_t^{10}}=2 \|U\|_{L_x^5L_{(T, \infty)}^{10}}\le \frac{\epsilon}{2}.
\end{equation}
From this inequality and Corollary \ref{xavcor} we get
\begin{equation}\label{xav19}
\|U\|_{X_{(T;\infty)}} \le B \|U(T)\|_{H^{1}(\R)}.
\end{equation}

From Theorem \ref{main.th1} it follows that
\begin{equation}\label{xav20}
\sup_{ t_0 \in \mathbb{R} }\sup_{0 \le t \le T}\|u_{\omega, t_0}(t)-U(t)\|_{H^1(\R)} \to 0, \quad \textrm{as} \quad |\omega| \to \infty.
\end{equation}

Thus, if $|w|$ is sufficiently large,  the triangular inequality along with \eqref{xav20} gives
\begin{equation}
\begin{split}
\|S(t) u_{\omega, t_0}(T)\|_{L_x^5L_t^{10}} & \le \|S(t) u_{\omega, t_0}(T)-S(t) U(T)\|_{L_x^5L_t^{10}}+ \|S(t) U(T)\|_{L_x^5L_t^{10}} \\
& \le  \|u_{\omega, t_0}(T)-U(T)\|_{L_x^2} + \frac{\epsilon}{2}\\
& \le \epsilon.
\end{split}
\end{equation}

Therefore, Corollary \ref{xavcor}  implies that $u_{\omega, t_0}$ is global. Moreover,
\begin{equation}\label{xav21}
\sup_{t_0 \in \mathbb{R}} \|u_{\omega, t_0}\|_{L_x^5L_{(T, \infty)}^{10}}\le 2\epsilon, \quad \textrm{and} \quad \|u_{\omega, t_0}\|_{X_{ (T,\infty)}} \le B \|u_{\omega, t_0}(T)\|_{H^1(\R)},
\end{equation}
for $|w|$ sufficiently large.

Let $M_0= \sup_{0\le t \le T}\|U(t)\|_{H^1(\R)}$, as in \eqref{eq4.1}. Now, we move to prove \eqref{eq1.19}. The inequalities \eqref{xav20} and \eqref{xav21} show that there exists $L>0$ such that
\begin{equation}\label{22}
\sup_{|w| \ge L}\sup_{t_0 \in \mathbb{R}}\sup_{t\geq 0}\|u_{\omega, t_0}(t)\|_{H^1(\R)} \le (1+M_0)+B\|u_{\omega, t_0}(T)\|_{H^{1}(\R)}=M_1<\infty.
\end{equation}

In what follows, we prove that $u_{\omega, t_0} \to U$ in the $\|\cdot\|_{X_t}$-norm, when $|\omega| \to \infty$.

Using  Duhamel's formulas for $u_{\omega, t_0}$ and $U$ we have
\begin{equation}\label{eq4.16}
\begin{split}
u_{\omega, t_0}(T+t) -U(T+t)&= S(t)(u_{\omega, t_0}(T) -U(T))\\
& \quad -\int_0^t S(t-t')g(\omega(T+t'+t_0))\partial_x (u_{\omega, t_0}^{k+1})(T+t') dt'\\
& \quad +m(g)\int_0^t S(t-t') \partial_x(U^{k+1})(T+t') dt'\\
&=:  I_1+I_2+I_3.
\end{split}
\end{equation}

 Using properties of the unitary group $S(t)$ we have by \eqref{xav20} that
\begin{align}\label{1xav}
 \|I_1\|_{X_t} =\|S(t)(u_{\omega, t_0}(T) -U(T))\|_{X_t} \le C \|u_{\omega, t_0}(T) -U(T)\|_{H^1(\R)} \stackrel{|\omega| \to \infty}{\rightarrow} 0.
\end{align}

 With the same argument as in \eqref{xav27}, we have
\begin{equation}\label{meq.1}
\|I_2\|_{X_t} \le C A\|u_{\omega, t_0}\|_{L_x^5L_{(T, \infty)}^{10}}^4\|u_{\omega, t_0}\|_{X_{(T, \infty)}},
\end{equation}

From \eqref{meq.1}, with the use of \eqref{xav21} and \eqref{22}, we have
\begin{equation}\label{2xav}
\|I_2\|_{X_t} \le  C A(2\epsilon)^4BM_1.
\end{equation}

As in $I_2$, using  \eqref{xav33} and \eqref{xav19},  we get
\begin{equation}\label{3xav}
\begin{split}
\|I_3\|_{X_t}&\le  %C A\|U\|_{L_x^5L_{(T, \infty)}^{10}}^5+ 
 C A\|U\|_{L_x^5L_{(T, \infty)}^{10}}^4\| U\|_{X_{(T, \infty)}}\\
& \le C A\Big(\frac\epsilon4\Big)^4BM_0.
\end{split}
\end{equation}

Now given $\beta>0$, we choose $\epsilon>0$ sufficiently small ($T$ sufficiently large) such that $C A(2\epsilon)^4\big[BM_0+BM_1 \big] < \beta/3$ and $|\omega|$ sufficiently large, so that \eqref{eq4.16}, \eqref{1xav}, \eqref{2xav} and \eqref{3xav} imply
\begin{equation}\label{eq4.19}
\begin{split}
\| u_{\omega, t_0}(t) -U(t) \|_{X_{(T, \infty)}} 
 &=  \|u_{\omega, t_0}(T+t) -U(T+t)\|_{X_t}\\
 & \le  \|I_1\|_{X_t}+\|I_2\|_{X_t}+\|I_3\|_{X_t}\\
 & <  \beta.
 \end{split}
 \end{equation}

On the other hand, from Theorem \ref{main.th1}, we have
\begin{equation}\label{eq4.20}
\| u_{\omega, t_0}(t) -U(t) \|_{X_{(0,T)}}=\| u_{\omega, t_0}(t) -U(t) \|_{X_{T}} \stackrel{|\omega| \to \infty}{\rightarrow}0.
\end{equation}

Therefore, from \eqref{eq4.19} and \eqref{eq4.20}, we can conclude the proof of the theorem.
\end{proof}
%%%%%%%%%%%%%%%%%%%%%%%%%%%%%%%%%%%%%%%%%%%%%%%%%%

\end{document}